\documentclass[11pt]{article}
\usepackage[ansinew]{inputenc}

\usepackage{latexsym,amsmath,amsthm,amssymb,amsfonts,amscd}
\usepackage{pstricks,pst-node,pst-plot}

\usepackage{a4wide}

\usepackage{graphicx}
\usepackage{color}
\usepackage[all]{xy}

\usepackage{graphics}

\date{Draft version --- \today}

\newcommand{\indic}[1]{\mathbf{1}_{#1}}
\newcommand{\indica}[1]{\mathbf{1}_{\{#1\}}}

\newtheorem{theorem}{Theorem}
\newtheorem{conj}[theorem]{Conjecture}
\newtheorem{lemma}[theorem]{Lemma}
\newtheorem{coro}[theorem]{Corollary}

\theoremstyle{definition}
\newtheorem{defn}[theorem]{Definition}
\newtheorem{remark}[theorem]{Remark}

\newcommand{\eps}{\varepsilon}
\newcommand{\E}{\mathbb{E}}
\renewcommand{\P}{\mathbb{P}}
\newcommand{\R}{\mathbb{R}}

\newcommand{\FF}{\mathcal{F}}
\newcommand{\bT}{\mathbf{T}}
\newcommand{\La}{\Lambda}

\newcommand{\N} {\mathbb{N}}

\newcommand{\eqlaw}{\overset{d}{=} }

\def\cP{\mathcal{P}}

\def\cE{\mathcal{E}}
\def\cD{\mathcal{D}}

\def\fg{g}
\def\bes{\bar{\alpha}}

\newcommand{\captionfonts}{\small}

\makeatletter  
\long\def\@makecaption#1#2{%
  \vskip\abovecaptionskip
  \sbox\@tempboxa{{\captionfonts #1: #2}}%
  \ifdim \wd\@tempboxa >\hsize
    {\captionfonts #1: #2\par}
  \else
    \hbox to\hsize{\hfil\box\@tempboxa\hfil}%
  \fi
  \vskip\belowcaptionskip}
\makeatother   

\begin{document}

\author{Julien Berestycki$^1$, Nathana\"{e}l Berestycki$^2$, and Vlada Limic$^3$}

\title{Asymptotic sampling formulae for $\Lambda$-coalescents}

\maketitle

\begin{abstract}
\noindent
We present a robust method which translates information on the speed of coming down from infinity of a genealogical tree into sampling formulae for the underlying population.
We apply these results to population dynamics where the genealogy is given by a $\La$-coalescent. This allows us to derive an exact formula for the asymptotic behavior of the site and allele frequency spectrum and the number of segregating sites, as the sample size tends to $\infty$.
Some of our results hold in the case of a general $\La$-coalescent that comes down from infinity, but we obtain more precise information under a regular variation assumption.
In this case, we obtain results of independent interest for the time at which a mutation uniformly chosen at random was generated. This exhibits a phase transition at $\alpha=3/2$, where $\alpha \in(1,2)$ is the exponent of regular variation.
\end{abstract}

\noindent {\em AMS 2000 Subject Classification.}
60J25, 60F99, 92D25

\medskip \noindent
{\em Key words and phrases.}
$\Lambda$-coalescents, speed of coming down from infinity, exchangeable coalescents, sampling formulae, infinite allele model, genetic variation.

\vfill
{\small \noindent
1. Universit\'e Pierre et Marie Curie - Paris VI. Supported in part by ANR MAEV and ANR MANEGE. \\
2. University of Cambridge. Supported in part by EPSRC grant EP/GO55068/1 and EP/I03372X/1. \\
3. Universit\'e de Provence. Research supported in part by
NSERC Discovery Grant, by Alfred P. Sloan Research Fellowship, and by ANR MAEV research grant.}


\clearpage


\section{Introduction and main results}
Coalescents with multiple collisions,
also known as {\em $\Lambda$-coalescents}, are a class of Markovian coalescence models, introduced and first studied by Pitman \cite{pit99} and independently by Sagitov \cite{sag99}, and were
already implicit in a contemporaneous work of Donnelly and Kurtz in \cite{dk99}.
They arise naturally as scaling limits for the genealogy of exchangeable population dynamics. This connection to population genetics has motivated a large number of works around the study of $\La$-coalescents, an introduction to which may be found in the recent surveys
\cite{bertoin, ensaios} for instance.

The following question is natural in the context of population genetics: assuming that the genealogical tree of a sample is a $\Lambda$-coalescent (definitions will be given below), how much genetic variation do we expect to see? A complete answer exists in the special case where only pairwise collisions are possible, due to the celebrated Ewens sampling formula~\cite{ewens_sam} for the Kingman coalescent~\cite{king82b}.
More recently, partial results have been obtained by Berestycki et al.~\cite{bbs2} in the particular case of Beta-coalescents (for a general overview of previous results on the subject, we refer the reader to Durrett~\cite{durrettDNA} or Berestycki~\cite{ensaios}.)

\medskip The main goal of this paper is to address this question in general, by providing a robust method which translates information about the \emph{speed of coming down from infinity} of the genealogical tree into an explicit asymptotic formulae (as the sample size increases to $\infty$) for
quantifying the genetic variation. Since the speed of coming down from infinity was recently analyzed by the authors in \cite{bbl1}, this method in combination with results from \cite{bbl1} enables us to obtain the following results:

(a) In Theorem \ref{T:A_n}, we obtain in complete generality (i.e., for arbitrary finite measures $\Lambda$ such that the corresponding $\La$-coalescent comes down from infinity) a deterministic asymptotic rate of growth for the number of distinct alleles in a sample and the number of segregating sites (the famous SNP count, or single nucleotide polymorphism). The formula involves a certain function $\psi(q)$, which is the Laplace exponent of the subordinator whose L\'evy measure is precisely $x^{-2} \La(dx)$. Furthermore, the above convergence in probability is strengthened to an almost sure convergence, provided that the measure $\Lambda$ satisfies an additional regular variation condition in the neighborhood of zero (precise assumptions will be given below).

(b) In Theorem \ref{T:spectrum} we derive explicit almost sure asymptotic formulae for the \emph{frequency spectrum}, in both the infinite site and the infinite allele models, in the case where $\La$ is regularly varying near zero.

\medskip This last result is a significant improvement and a generalization of previous work of Berestycki et al.~\cite{bbs1, bbs2} and of Schweinsberg~\cite{Schweinsberg2009}, both in the sense that the result is valid for
more general measures $\La$, and in the sense that the convergence holds almost surely rather than in probability.
Our methodology is completely different from that of~\cite{bbs2} which relied on an embedding into stable continuous random trees (CRT), allowing for explicit computations. As explained above, the argument here is based on the recent work by the authors on the speed of coming down from infinity~\cite{bbl1}, and a novel general method which translates such results into results about sampling formulae.
This method is more robust than previous approaches to this problem, which explains why the results here are both stronger and more general.

We note that the asymptotics in probability for the number of distinct alleles in a sample, under the model where the genealogy is driven by the general (regular) $\Xi$-coalescent dynamics,  was obtained in parallel by Limic~\cite{sampl_xi} using an adaptation of the martingale method that led to the results in \cite{bbl1} and \cite{scdi_xi} (see Remark \ref{R:xi case}).

In the sequel, we denote by $\Rightarrow$ the convergence in distribution, and by $\eqlaw$ the equivalence in distribution.
We also use the standard Bachmann-Landau notation $\sim, O(\cdot), o(\cdot),\asymp$ for comparing asymptotic behavior of deterministic and stochastic functions and sequences.

\subsection{Mutation models}
\label{S:asym samp}

We now describe the underlying framework for the sampling results in more detail.
Consider a sample of $n$ individuals, where $n$ is a fixed number tending to infinity. Assume that the genealogical relationship between these individuals is given by a $\Lambda$-coalescent, where $\Lambda$ is an arbitrary finite measure on $[0,1]$. That is, the genealogical tree is a Markov process $(\Pi^n(t), t \ge 0)$ on the space of partitions of $\{1, \ldots, n \}$ with the following transition rates: whenever $\Pi^n$ has $b$ blocks, any $k$-tuple of them merges at rate $\lambda_{b,k}: = \int_0^1 x^{k-2}(1-x)^{b-k} \Lambda(dx)$.

In order to discuss genetic variation, we need to specify a mutation model. The two most widely used and tractable models are the \emph{infinite sites model} and the \emph{infinite alleles model}. To familiarize oneself with these models, it is also useful to think in terms of the forward-in-time evolution dynamics for the whole population (and not only in terms of the backward-in-time coalescent dynamics).

In the classical \emph{infinite sites model},
introduced by Kimura~\cite{kimura} in 1969,
any individual is affected by neutral mutations at constant rate $\theta>0$.
Here it is also
assumed that the number of loci (the size of the genome) is large, so that each mutation occurs at
a new locus. In particular, if an individual is affected by a
mutation, then \emph{all} the descendants of this individual carry
this mutation (see Figure 1). Conversely, the genetic type of any individual in the sample depends on the entire history of its ancestral lineage. We denote by
$S_n$ the {\em number of segregating sites}, or the total number of distinct genetic types in the sense just described. This is the same as the omnipresent SNP count (single nucleotide polymorphism) from the biological literature.

The \emph{infinite alleles model} is similar but with one difference.
As above, any individual is affected by a mutation at constant rate $\theta$. However, it is now assumed that every mutation changes the allelic type of the individual into something new, distinct from anything else (already seen or yet unseen) in the population. Thus, the allelic type of an individual in the sample is entirely determined by the most recent mutation affecting the corresponding ancestral lineage.
The {\em allelic partition} is the partition of the sample (represented by a partition of $\{1, \ldots, n\}$) obtained by grouping together the individuals that carry the same allelic type. Denote by $A_n$ the number of blocks in this partition, or equivalently, the total number of allelic types expressed in the sample.

\begin{remark}
  The two models differ only by the amount of information that is assumed to be available in the sample. In the infinite sites model the assumption is that the precise allelic type (e.g., the entire DNA sequence) is known for each individual in the sample. On the other hand,
  in the infinite alleles model, the only available information is whether two individuals carry the same type or not. Hence for different types we do not know how they differ.

  Thus the infinite alleles model contains less information than the infinite sites model, and is more appropriate in practice for situations where the only available information is, for example, based on observed physiological differences. On the other hand the infinite sites model is more natural when the full genetic information (the DNA sequences of each individual in the sample) is available.
\end{remark}

\medskip The above random variables can be realized in a natural way on a common probability space as follows. Consider a $\Lambda$-coalescent $(\Pi_t,t\ge0)$ that comes down from infinity,
and let $\bT$ be the associated coalescent tree. Then $\bT$ is a tree with infinitely many leaves $1, 2, \ldots$, and the root given by the most recent common ancestor among all the individuals.
Each branch of $\bT$ is endowed with a positive number, its length or the size of the interval of time that elapsed between the two defining coalescent events for this branch (the one that started and the one that ended it) .
Let $\cP$ be
a Poisson process of mutations on the branches of $\bT$, where the intensity of mutations is
constant and equal to $\theta$ per unit length.
Restricting $\bT$ to the first $n$ leaves produces a finite tree  (even if the $\La$-coalescent does not come down from infinity), denoted by $\bT_n$, that has the law generated by the same $\Lambda$-coalescent started from $n$ particles. The restriction of $\cP$ to $\bT_n$ is identified as
the mutation process on $\bT_n$, and it is a sufficient statistic for
$S_n$ and $A_n$.
It is useful to note here that $\bT$ and $\cP$ alone determine, simultaneously for all $n$, the values of $A_n$ and $S_n$, as well as various related quantities to be introduced in the sequel.
Moreover, the coupling induced by this procedure between $\bT_n$ and $\bT_m$ for $m<n$, is canonical from the sampling perspective, in that the mutations that arrive onto $\bT_n$ also arrive onto $\bT_m$.
On the asymptotically unlikely event $\{\cP \cap \bT_n = \emptyset\}$, we declare $A_n = S_n = 0$.

\begin{center}
\includegraphics{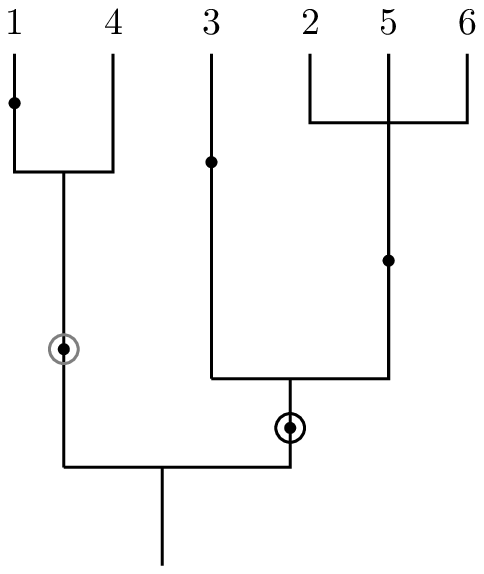}
\end{center}

\noindent
{\small Figure 1:
The genealogical tree $\bT_n$ for a sample of size $n = 6.$
The mutations on the (vertical) branches of $\bT_n$ are indicated as dots.
The dot encircled in black corresponds to the mutation which
is not seen under the infinite alleles model.
The dot encircled in gray will be referred to in Section \ref{S:sampl_formul}.
Thus we have $S_n = 5$ while $A_n=4$.}
\setcounter{figure}{1}

\subsection{Sampling formulae}
Let $\Lambda$ be a finite measure on $[0,1]$.
We will assume without further mention that $\La({1}) = 0$
(for reasons why this can be done without loss of full generality see any of~\cite{bbl1,pit99,sch1}).
\begin{defn}\label{D:regvar}
We say that $\Lambda$ has (strong) $\alpha$-regular variation at zero if
$\Lambda(dx)= f(x)dx$ where
$f(x) \sim A x^{1- \alpha}$ as $x \to 0$ for some $1< \alpha <2$ and $A>0$.
\end{defn}

For any given finite measure $\Lambda$ on $[0,1]$, associate a function $\psi_\La=\psi$ defined by
\begin{equation}
\label{D:psi} \psi(q):= \int_{[0,1]}(e^{-qx}-1+qx)x^{-2}\Lambda(dx).
\end{equation}
The function $\psi$ is the Laplace exponent of a L\'evy process, which is intimately connected with the behavior of the $\La$-coalescent. These links are discussed in a companion paper \cite{particle_new}.

In this paper we will usually require that
\begin{equation}
  \label{E:grey0}
  \int_1^\infty \frac{dq}{\psi(q)} < \infty,
\end{equation}
which is known as {\em Grey's condition}. As was proved by Bertoin and Le Gall~\cite{blg3} (see also \cite{particle_new} for a probabilistic proof),
this is equivalent to the requirement that the $\La$-coalescent comes down from infinity.
One can check (see e.g. \cite{FellerV2} XIII.6) that in the case of strong $\alpha$-regular variation,
\begin{equation}
\label{Epsi strongreg}
\psi(q) \sim \frac{A \Gamma(2-\alpha)}{\alpha(\alpha -1)} q^\alpha, \mbox{ as $q \to \infty$},
\end{equation}
where $A$ is the constant from Definition \ref{D:regvar},
so in particular the Grey condition (\ref{E:grey0})
holds if $\alpha\in (1,2)$. Our first result concerns the asymptotic behavior of the number $S_n$ of segregating sites and the size $A_n$ of the allelic partition.

\begin{theorem} \label{T:A_n} Assume (\ref{E:grey0}) and let
 $X_n$ denote either $A_n$ or $S_n$. Then
\begin{equation}\label{E:An general}
\frac{X_n }
{\displaystyle \int_1^n {q}{\psi(q)^{-1}}\,dq }\longrightarrow \theta,
\end{equation}
in probability as $n \to \infty$.
Moreover, if $\Lambda$ has (strong) $\alpha$-regular variation  at zero,
then the above convergence holds almost surely, implying
\begin{equation}\label{E:TAn-strong}
n^{\alpha-2}{X_n} \longrightarrow \theta B, \mbox{ almost surely,}
\end{equation}
where $B = B(A, \alpha) := \alpha(\alpha -1) / [A \Gamma(2-\alpha)(2-\alpha)]$.
\end{theorem}

\begin{conj}
Though our proofs rely on the fact that the $\La$-coalescent comes down from infinity, we observe that the above statement does not appeal to the function $v(t)$ and hence could hold in general.
Due to the result of Basdevant and Goldschmidt \cite{BasdevantGoldschmidt}
(see also Remark 8 in \cite{sampl_xi} for a short and more robust argument) on the number of allelic families in the Bolthausen-Sznitman coalescent,
it is easy to check that \eqref{E:An general} holds in this special case, whereas \eqref{E:grey0} does not hold.
We conjecture that \eqref{E:An general} holds for all the $\La$-coalescents even in the $L^1$ convergence sense.
\end{conj}

\begin{remark} Let $\bar\psi(q):=\int_{[0,1]} ((1-x)^q -1 + qx)/x^2\,\La(dx)$ be a close relative of $\psi$ (note that $|\bar\psi(q)- \psi(q)|=O(q)$ and moreover that $\bar\psi(q)\sim \psi(q)$ as $q\to \infty$).
Applying the optional stopping formula to the following submartingale
\[
\left(\int_{N^{\La, n}(t)}^n \frac{q}{\bar\psi(q)}\,dq - \int_0^t N^{\La, n} (u)\,du,\ t\geq 0\right)
\]
gives an expectation upper bound in support of the above conjecture.
We refer the reader to (\ref{E cdi}) for the rest of the notation, and to the above mentioned remark in
\cite{scdi_xi} as well as the argument leading to (25) in \cite{vl_habi} for applications of similar processes in the study of coming down from infinity.
\end{remark}

\medskip
We also obtain precise results for the full \emph{frequency spectrum} in both infinite sites and infinite alleles models. For each $n$ consider a sample of size $n$, and for each $k\in \{1,\ldots,n\}$, let $F_{k,n}$ be the number of families of size $k$ in its allelic partition, and $M_{n,k}$ be the number of mutations affecting precisely $k$ of its individuals  under the infinite sites model.

\begin{theorem} \label{T:spectrum}
Suppose that $\Lambda$ has (strong) $\alpha$-regular variation  at zero.
Recall the constant  $B\equiv B(\alpha,A)$  from Theorem \ref{T:A_n}.
Let $X_{k,n}$ denote $M_{k,n}$ or $F_{k,n}$, where $1\le k\le n$. As $n\to \infty$,
\begin{equation}\label{E:allelic spectrum}
\frac{X_{k,n}}{n^{2-\alpha}} \rightarrow \theta B \frac{(2-\alpha) \Gamma(k+\alpha -2)}{k!\Gamma(\alpha-1)} ,\ \ a.s.
\end{equation}
Moreover, if $P_1, P_2,\ldots$ are the ordered allele frequencies in the population, then
\begin{equation}\label{E:Tfreq}
  P_j \sim C j^{-1/(2-\alpha )},
\end{equation}
almost surely as $j \to \infty$, and $C = (\theta B/\Gamma(\alpha -1))^{1/(2-\alpha)}$.
\end{theorem}
By the properties of the Gamma function, another expression for the constant on the right-hand side of \eqref{E:allelic spectrum} is
$$
\theta B (2-\alpha) \frac{(\alpha -1) \ldots (\alpha +k-3)}{k!}.
$$
As mentioned in the Introduction,
the above results are improvements over previously known results, since
the convergence \eqref{E:allelic spectrum} was known to hold only in probability in the case of Beta-coalescents (see~\cite{bbs2}), while \eqref{E:Tfreq} was not known to hold even in this special case.

One key ingredient for our arguments, beside our earlier work on the speed of coming down from infinity (\cite{bbl1}) is the asymptotic study of
the time at which a randomly chosen (uniformly) mutation was generated.
More precisely, let the time run in the ``coalescent direction'' (backward from the point of view of the population dynamics), so that the leaves of the tree are present
at time $0$, and the number of branches decreases in time.
Denote by $M_n$ the time-coordinate (age) of a point chosen at random from $\cP \cap \bT_n$.
On the (asymptotically unlikely) event $\{\cP \cap \bT_n = \emptyset\}=\{S_n=A_n=0\}$, we set $M_n=0$
(although this value could be set to anything between $0$ and the time of MRCA (the root of $\bT_n$) and the next result would still be true).

Define
\begin{equation}
\label{Efas}
\fg(n) = \begin{cases}
  n^{1- \alpha}, & \text{ if } 1<\alpha < 3/2\,, \\
  n^{-1/2} \log n, & \text{ if } \alpha = 3/2\,,\\
  n^{\alpha-2}, & \text{ if } 3/2 < \alpha <2\,.
\end{cases}
\end{equation}
\begin{theorem}\label{P:M_n}
Suppose that $\Lambda$ has (strong) $\alpha$-regular variation at zero, for some $\alpha\in (1,2)$.\\
(a) We have
\begin{equation}
\label{eq:M_n a}
\frac{M_n}{n^{1-\alpha}} \Rightarrow \,\frac{\alpha}{A \Gamma(2-\alpha)}\, U^{-\frac{\alpha-1}{2-\alpha}} - 1,
\end{equation}
where $U\eqlaw\ Unif[0,1]$.

\noindent (b) If in addition $\La[1-\eta,1]=0$ for some $\eta>0$, then
there exists $c_1\equiv c_1(\alpha)\in (0,\infty)$, such that for $\fg$ given by (\ref{Efas})
\begin{equation}
\label{eq:M_n b}
\lim_{n\to \infty} \frac{\E(M_n)}{\fg(n)}=c_1.
\end{equation}
\end{theorem}

\begin{remark}
Note that $M_n$ and $\E(M_n)$ are only of the same order of magnitude if $\alpha < 3/2$.
Naturally, this is because the limit variable is integrable if and only if $\alpha <3/2$.

Interestingly, $\fg(n)$ observed as a function of $\alpha$ decreases on
$(1,3/2)$ and increases on $(3/2,2)$, and moreover has a discontinuity on both sides at $\alpha = 3/2$. It seems difficult to see intuitively why this happens.
\end{remark}

\begin{remark}
We believe that the result \eqref{eq:M_n b} should hold without any further restriction on $\Lambda$ than strong regular variation. The techniques used in the proof of (10) can be used to show, with some additional effort, that the sequence $\E(M_n) / g(n)$ is bounded away from 0 and infinity when no assumption is made on the support of $\Lambda$. However, in the interest of brevity we decided to omit these arguments.
\end{remark}

The link between Theorem \ref{T:A_n} and Theorem \ref{T:spectrum} is provided by a remarkable Tauberian theorem for random partitions of Gnedin, Hansen and Pitman~\cite{ghp}. The assumptions
of this theorem were recently extended in an independent but related work of Schweinsberg~\cite{Schweinsberg2009}, to deal with convergence in probability (to which the approach of~\cite{ghp} could not apply). This allowed him to obtain the convergence in probability of Theorem \ref{T:spectrum} for
the limiting behavior of $F_{k,n}$ (though not for that of $M_{k,n}$).
Gnedin, Hansen and Pitman~\cite{ghp} also derive a central limit theorem for $F_{k,n}$.
It is natural to ask whether this result can be extended to our setting with random frequencies.
Kersting \cite{kersting} has recently obtained precise fluctuation results for the length of the genealogical tree in the regularly varying case.
In particular, it follows from his Theorem 1 that these fluctuations are not Gaussian. In order to resolve the just mentioned open problem, one would need to analyze the complex interplay between the fluctuations of the tree length and the Poisson fluctuations of the mutations.

\medskip \textbf{Organisation of the rest of the paper.} Section \ref{S:sampl_formul} is devoted to proving the results on the mutation frequency spectrum,
announced in Section \ref{S:asym samp}. More precisely, we prove
Theorem \ref{P:M_n} in Section \ref{S:proof of Mn},
Theorem \ref{T:A_n} in Section \ref{S:proof of An}, and Theorem \ref{T:spectrum} in Section \ref{S:proof of T:spectrum}.
The final section relaxes the technical condition on the support of $\La$ needed in the proofs of Theorems \ref{T:A_n} and \ref{T:spectrum}.

\section{Proofs of the results} 
\label{S:sampl_formul}
Fix some $\theta >0$.
For each $n\in \N$ and $t\ge 0$, let $N^{\La, n} (t)$ denote the number of ancestral lineages of the first $n$ individuals remaining at time $t$.
In particular, $(N^{\La, n} (t),\,t\geq 0)$ is a continuous-time Markov jump process, starting from
$N^{\La, n} (0)=n$.
We assume throughout this section that (\ref{E:grey0}) holds, or that equivalently, the
$\La$-coalescent comes down from infinity:
\begin{equation}
\label{E cdi}
N^{\La}(t) := \lim_{n\to \infty} N^{\La,n}(t) = \sup_{n\geq 1} N^{\La,n}(t) <\infty, \ \forall t>0.
\end{equation}
We will need some further notations.
Define for $k,n\in \N$,
\begin{equation}
\label{Etau k n}
\tau_k^n= \inf \{ t \ge 0 : N^{\La, n} (t) \le k\}, \ \mbox{ and } \tau_k\equiv \tau_k^\infty  = \lim_{n \to \infty} \tau_k^n = \inf \{ t \ge 0 : N^{\La} (t) \le k\}.
\end{equation}
In particular
$\tau_1^n= \inf \{ t \ge 0 : N^{\La, n} (t) =1\}$
is the time of the MRCA for the sample containing the first $n$ individuals, and $L_n= \int_{0}^{\tau_1^n} N^{\La,n}(t)\,dt$ is the total length of the tree $\bT_n$.

The genealogical tree $\bT$ is a path-connected set in $\R^2$.
To any point $x$ on the tree one can associate a number $t=t(x)$ called the {\em time-coordinate} or the {\em age} of $x$, which is defined as the distance from that point to the set of leaves of $\bT$.
Let $\hat \bT_n$ be the subtree of $\bT$ consisting of all the
points in $\bT$ having age in $[\tau_n,\tau_1]$.
Then $\hat L_n = \int_{\tau_n}^{\tau_1} N^\La(u)du$ is the length of $\hat \bT_n$.  The function
\begin{equation}\label{D:v}
v(t) := \inf\{ s \ge 0 : \int_s^\infty {dq \over \psi(q)} < t\}
\end{equation}
plays a central role in the analysis of asymptotic behavior of $\hat \bT_n$.
Finally, define $t_n$ as
\begin{equation}
\label{D:tn}
t_n:= \int_{n}^\infty \frac{dq}{\psi(q)}\equiv v^{-1}(n).
\end{equation}

The following lemma gathers some asymptotic results which we will use in the rest of the proof.
Set
\begin{equation}
\label{Ebes}
\bes  = \frac{2-\alpha}{\alpha-1}\in (0, \infty).
\end{equation}

\begin{lemma}\label{L:asymptotics}
Assume (\ref{Epsi strongreg}) and
define $c=c(A,\alpha)= \alpha/(A\Gamma(2-\alpha))$ (compare with the constant in (\ref{eq:M_n a})).
Then, as $n\to \infty$, we have almost surely
\begin{enumerate}
\item
$\tau_n \sim c\, n^{1-\alpha}$, and
\item
$t_n \sim c\, n^{1-\alpha}$.
\end{enumerate}
Furthermore, there exist $c_1=c(\La)>0$,  and $c_2=c(A,\alpha)\in (0,\infty)$, such that
\begin{enumerate}
\setcounter{enumi}{2}
\item
$\P(\tau_1>x) \le e^{ -c_1 \cdot x}$, for all $x\geq 1$,
\item
$\hat L_n= \int_{\tau_n}^{\tau_1} N^{\La}(u) \,du \sim \int_{t_n}^1 v(u) \,du   \sim \frac{c^{\frac{1}{\alpha-1}}}{\bes} (\tau_n)^{-\bes} \sim \frac{c}{\bes} n^{2-\alpha}$, a.s.~as $n\to \infty$,
\item
as $x \to 0$
\begin{equation}
\int_{x}^{\tau_1} u \, N^{\La}(u) \,du \sim \left\{
\begin{array}{ll}
  c_2 x^{-\bes +1 }, &  \bes >1, \\
  c_2 \log (1/x), &   \bes =1, \\
   Y, &  \bes <1,
\end{array}
\right. \mbox{a.s. and }
\int_{x}^{1} u \, v(u) \,du \sim \left\{
\begin{array}{ll}
  c_2 x^{-\bes +1 }, &  \bes >1, \\
  c_2 \log (1/x), &   \bes =1, \\
   c_2, & \bes <1,
\end{array}
\right.
\end{equation}
where $c_2=c_2(\alpha)$ and $Y := \int_0^{\tau_1} uN^\La(u)du$ is a finite random variable if $\bes <1$.
\end{enumerate}
\end{lemma}

\begin{proof}
Theorem 1 in \cite{bbl1} and (\ref{Epsi strongreg}) yield
\begin{equation}
\label{ENasym streg}
N^\La(t) \sim v(t) \sim c^{\frac{1}{\alpha-1}} t^{-\frac{1}{\alpha-1}}, \mbox{ as } t \to 0, \mbox{ almost surely},
\end{equation}
where $c=c(A,\alpha)$ is as specified above.
The asymptotic behavior (\ref{ENasym streg}) implies that
$N^\La(\tau_n) = n(1+o(1))$, almost surely, as $n\to \infty$.
Indeed, since $\tau_n \to 0$, we have $N^\La(\tau_n)\sim v(\tau_n) = v(\tau_n-) \sim N^\La(\tau_n-)$, and at the same time, $\P(N^\La(\tau_n)\leq n < N^\La(\tau_n-))=1$.
Since, again due to (\ref{ENasym streg}) $N^\La(\tau_n) \sim c^{\frac{1}{\alpha-1}} (\tau_n)^{-\frac{1}{\alpha-1}}$, we obtain claim 1.
Claim 2 is directly seen from (\ref{D:tn}) and the asymptotic behavior of $v$ in (\ref{ENasym streg}).

Claims 4 and 5 are derived similarly.
One notes first that, due to $\P(\tau_1>0)=1$ and (\ref{ENasym streg}),
\begin{align*}
\int_{x}^{\tau_1} N^{\La}(u) \,du \sim \int_{x}^{1} v(u) \,du & \sim
\frac{c^{\frac{1}{\alpha-1}}}{\bes}  x ^{-\bes} \ \mbox{ as $x\to 0$,}
\end{align*}
and then uses the facts that $\tau_n \to 0$, $t_n\to 0$ as well as claims 1 and 2 to obtain claim 4.
For claim 5,
we use  $u N^\La(u) \sim c^{\frac{1}{\alpha-1}}  u^{-1/(\alpha -1) + 1} =  c^{\frac{1}{\alpha-1}} u^{-\bes}$, and this uniquely determines $c_2$. If $\bes <1$, then both
$\int_0^1 uv(u) \,du$ and $\int_0^{\tau_1} u N^\La(u) \,du$ are finite.

It remains to verify claim 3. Due to monotonicity of the coalescent and the simple Markov property, we have
$\P(\tau_1^n \geq m+1|\tau_1^n \geq m) \leq \P(\tau_1^n \geq 1)$. In turn, letting $n \to \infty$,
\[
\P(\tau_1 \geq m+1|\tau_1 \geq m) \leq \P(\tau_1 \geq 1), \mbox{ for each }m\geq 1.
\]
Hence, by induction, $\P(\tau_1 > m) \le e^{-c_1m}$ for all $m \ge 0$, with $c_1 = \log \P(\tau_1 \geq 1)$.
\end{proof}

\subsection{Proof of Theorem \ref{P:M_n}}
\label{S:proof of Mn}
Our first goal will be to prove Theorem \ref{P:M_n}.
Let $A_1^n=\{S_n\geq 1\}=\{A_n\geq 1\}$.
Since the length $L_n$ of $\bT_n$ diverges, we have that $\P(A_1^n)\to 1$, as $n\to \infty$.
Recall that on $A_1^n$, $M_n$ is the age of a randomly chosen mutation in $\bT_n$, and that on the complement of $A_1^n$, $M_n$ is set to $0$.
Due to basic properties of Poisson point processes, on the event $A_1^n$ (of overwhelming probability),
the random mutation is positioned as a point $P_n^*$ chosen uniformly at random from $\bT_n$.
In symbols,
\begin{equation}
\label{E M express}
M_n = M_n^* \cdot \indic{A_1^n} + 0 \cdot \indic{(A_1^n)^c},
\end{equation}
where $M_n^*$ is the time-coordinate of $P_n^*$, and $P_n^*$ is independent of
$A_1^n$.
Due to this independence, and the fact $\P(A_1^n)\to 1$ as $n\to \infty$, we have that
$\E[M_n^* \cdot \indic{(A_1^n)^c}] = \E(M_n^*)\P((A_1^n)^c)=o(\E(M_n^*))$, and therefore $\E(M_n) \sim \E(M_n^*)$, as $n\to \infty$.
Similarly
\[
\P(M_n \cdot n^{\alpha-1} \leq x) =
\P(M_n^* \cdot n^{\alpha-1} \leq x) + O(\P((A_1^n)^c)), \ \forall x\geq 0,
\]
therefore (\ref{eq:M_n a}) is equivalent to
\begin{equation}
  \label{eq:Mnstar}
  \frac{M^*_n}{n^{1-\alpha}} \Rightarrow \,c\, (U^{-(\alpha-1)/(2-\alpha)} - 1), \mbox{ where $U\eqlaw$\ Unif[0,1]},
\end{equation}
Hence we proceed by studying $M_n^*$.

Now recall the subtree $\hat \bT_n$ of $\bT$.
Consider a uniform random point on $\hat \bT_n$ and let $\hat M_n$ be its age {\bf minus} $\tau_n$.
Let $\cE_n = \{N^\La_{\tau_n} = n\}$ be the event that $\Pi$ ever attains a configuration with exactly $n$ blocks. Due to the consistency property and the Markov property of
 $(\Pi_t,t\ge 0)$,
\begin{equation}
\label{ET cond dist}
\mbox{ the conditional law of
 $\hat \bT_n$ given $\FF_{\tau_n}$ on the event $\cE_n$ equals the law of $\bT_n$.}
\end{equation}
This clearly induces the equivalence of the
conditional law of
 $\hat M_n$ given $\FF_{\tau_n}$ on $\cE_n$ and the law of $M_n^*$,
which will be used below.
Recalling (\ref{Etau k n}),
we have
\begin{equation*}
\P(\hat M_n \geq x | \hat \bT_n) = \frac{\int_{\tau_n +x}^{\tau_1} N^{\La}(u) \,du}{\hat L_n}, \ x\in [0,\tau_1 - \tau_n].
\end{equation*}
So, recalling $\bes$ from (\ref{Ebes}), we have for a fixed $y\in (0,1)$
\begin{equation}\label{E: sans les termes negligeables}
\P((\hat M_n/\tau_n +1 )^{-\bes} \leq y) = \E\left[  \frac{\int_{\tau_n y^{-1/\bes} \wedge \tau_1}^{\tau_1} N^{\La}(u) \,du}{\hat L_n} \right].
\end{equation}
By the argument used to show claim 4 of Lemma \ref{L:asymptotics} and $\P(\tau_1>0)=1$, we conclude
\begin{align*}
\int_{\tau_ny^{-1/\bes} \wedge \tau_1}^{\tau_1} N^{\La}(u) \,du & \sim c'[\tau_ny^{-1/\bes} \wedge \tau_1]^{-1/(\alpha-1) +1} \sim c'[\tau_ny^{-1/\bes}]^{-1/(\alpha-1) +1}
=c' \tau_n^{-\bes} y, \mbox{ a.s.}
\end{align*}
as $n \to \infty$, where $c'=c^{1/(\alpha-1)}/\bes$. Due to claim 4 of Lemma \ref{L:asymptotics} we now see that the random variable inside the expectation on the RHS of \eqref{E: sans les termes negligeables} converges to $y$, almost surely.
Since
$$
|\E[\P((\hat M_n/\tau_n+1)^{-\bes} \leq y \,|\,\FF_{\tau_n}) \indic{\cE_n}]
- y \P(\cE_n)| \le \E\left[ \indic{\cE_n} \left| \frac{\int_{\tau_n y^{-1/\bes}}^{\tau_1} N^{\La}(u) \,du}{\hat L_n}-y \right| \right],
$$
and since the bounded (by 2) random variable inside the expectation
converges to 0 almost surely, the dominated convergence theorem
implies that the left-hand side converges to 0.
Recalling (\ref{ET cond dist}), or its consequence for random points,
this quantity can be rewritten as
$
|\,\P((M_n^*/\tau_n+1)^{-\bes} \leq y) - y\,|\, \P(\cE_n).
$
Theorem 1.8 in \cite{bbs1} gives $\P(\cE_n) \to \alpha -1$ as $n \to \infty$,
hence
$
\P(( M_n^*/\tau_n+1)^{-\bes} \leq y) \to y$, for all $ y \in [0,1].
$
This, together with claim 1 of Lemma \ref{L:asymptotics}, implies \eqref{eq:M_n a}.

\emph{The proof of (\ref{eq:M_n b})} is analogous but technically more delicate. Due to
$\P(M_n^*\geq x| \bT_n)=\int_x^{\tau_1^n} N^{\La,n}(u)\,du/L_n$ and Fubini's theorem, we have
\begin{eqnarray}
\E(M_n^* | \bT_n) &= \displaystyle \frac{\int_0^\infty \int_{ x}^{\tau_1^n} N^{\La,n}(u) \,du \, dx}{L_n}
&= \frac{\int_{0}^{\tau_1^n} u \, N^{\La,n}(u) \,du}{\int_{0}^{\tau_1^n} N^{\La,n}(u) \,du}.\label{E M condexp}
\end{eqnarray}
Therefore, $\E(M_n^*)  = \E\left[ Y_n \right]$, where
\begin{align}
 Y_n = \frac{\int_{0}^{\tau_1^n} u \, N^{\La,n}(u) \,du}{\int_{0}^{\tau_1^n} N^{\La,n}(u) \,du}.
\label{E M condexp2}
\end{align}
Due to (\ref{ET cond dist}), the variable $Y_n$ is equal in law to
\begin{equation}
\label{Eeqllaw}
\hat Y_n = \frac{\int_{\tau_n}^{\tau_1} (u-\tau_n) \, N^{\La}(u)
\,du}{\int_{\tau_n}^{\tau_1} N^{\La}(u) \,du} =
\frac{\int_{\tau_n}^{\tau_1} u \, N^{\La}(u)
\,du}{\int_{\tau_n}^{\tau_1} N^{\La}(u) \,du} -\tau_n, \mbox{ given $\FF_{\tau_n}$,
on the event $\cE_n$.}
\end{equation}
Using Lemma \ref{L:asymptotics}
one can analyze the asymptotic behavior of ${\hat Y_n}/{\fg(n)}$
in each of the three cases  $\bes >1$, $\bes<1$ and $\bes=1$
(corresponding respectively to $\alpha < 3/2$, $\alpha>3/2$ and
$\alpha = 3/2$).
First observe that $\tau_n / \fg(n) \to 0 $ almost surely if $\alpha \ge
3/2$,
and otherwise $\tau_n / \fg(n) \to c $ almost surely, where $c$ is the
constant from claim 1 of Lemma \ref{L:asymptotics}.
One can apply claims 4 and 5 (plugging in $\tau_n$ as $x$) of Lemma \ref{L:asymptotics} to the first term in (\ref{Eeqllaw}).
More precisely, if we let $c_3=c_2\bes/c^{ \bes}$ and $c_4= \bes /c$, then
\begin{enumerate}
\item if $\bes >1$ then $\int_{\tau_n}^{\tau_1} u \, N^{\La}(u) \,du \sim c_2 \tau_n^{-\bes+1}$ so ${\hat Y_n}/{\fg(n)} \sim c_3-c$. almost surely.

\item if $\bes<1$ then $\int_{\tau_n}^{\tau_1} u \, N^{\La}(u) \,du \sim Y$ so ${\hat Y_n}/{\fg(n)} \sim c_4 Y $, almost surely.

\item if $\bes=1$ then $\int_{\tau_n}^{\tau_1} u \, N^{\La}(u) \,du \sim -c_2 \log \tau_n\sim c_2 (\alpha-1)\log n$  so ${\hat Y_n}/{\fg(n)} \sim (\alpha-1)c_2/c$, almost surely.
\end{enumerate}
With a slight abuse of notation, let $Y:=\lim_{n\to \infty} {\hat Y_n}/{\fg(n)}$, almost surely. Clearly $\P(Y\geq 0)=1$. Denote by $\cD_c(Y)$ the set of points of continuity for the distribution function of $Y$.
Then for any $x>0$ in $\cD_c(Y)$ and any sequence $(B_n)_{n\geq 1}$ of events, where $B_n\in \FF_{\tau_n}$, $n\geq 1$, we have
\begin{equation}
\label{Econdconve}
 \E(\indic{B_n} \E(|\indica{\hat Y_n/\fg(n) \leq x} - \indica{Y\leq x}|\,|\,\FF_{\tau_n}) )=o(1),
 \mbox{ as }n\to \infty.
\end{equation}
 We claim that
\begin{equation}\label{claimYn}
\frac{Y_n}{\fg(n)} \Rightarrow Y, \mbox{ as $n\to \infty$},
\end{equation}
which can be verified as follows. Note that, due to (\ref{ET cond dist}), for each fixed $x>0$
\[
\P\left(\frac{Y_n}{\fg(n)} \leq x\right) =
\frac{\E[\indic{\cE_n}\P(\hat Y_n/\fg(n) \leq x\,|\,\FF_{\tau_n}) ]}{\P(\cE_n)},
\]
Backward martingale convergence and measurability $Y\in \FF_0$ imply
$\lim_n \P(Y\leq x\,|\,\FF_{\tau_n})= \P(Y\leq x\,|\,\FF_0)= \P(Y\leq x)$.
Combined with (\ref{Econdconve}) and the fact $\liminf_n \P(\cE_n)>0$, this
gives $\lim_n \P\left(\frac{Y_n}{\fg(n)} \leq x\right)= \P(Y\leq x)$, for each $x\in \cD_c(Y)$, or equivalently, the convergence (\ref{claimYn}).

To conclude (\ref{eq:M_n b}) from \eqref{claimYn}, it thus suffices to show that $(Y_n/\fg(n))_{n\geq 1}$ is a uniformly integrable family.
In fact we will now show that this family is uniformly bounded in $L^2$.
Due to (\ref{E M condexp2}), we have  $\P(Y_n\leq \tau_1^n)=1$, and in particular
$$
Y_n \indica{ \tau_1^n \le \fg(n)} \le \fg(n), \mbox{ almost surely}.
$$
Due to claims 4 and 5 of Lemma \ref{L:asymptotics}  we know that $\fg(n) \sim c \int_{t_n}^1 u v(u) du / \int_{t_n}^1 v(u)du$ for some $c=1/C\in (0,\infty)$. Therefore
\begin{align}
\frac{Y_n}{\fg(n)} & \le C \frac{\displaystyle \int_{0}^{\tau_1^n} u \, N^{\La,n}(u) \,du}{\displaystyle \int_{t_n}^{1}  u v(u) du }
\cdot \frac{\displaystyle \int_{t_n}^{1} v(u)du }{\displaystyle \int_{0}^{\tau_1^n} N^{\La,n}(u) \,du}  \indica{ \tau_1^n > \fg(n)}  +1. \label{Eletit}
\end{align}
Denote by $A_n$ (resp,~$B_n$) the first (resp.~second) ration on the LHS of (\ref{Eletit}), so that
$\frac{Y_n}{\fg(n)}\le C A_n \cdot B_n \indica{\tau_1^n > \fg(n)}  +1$.
We will bound separately the terms $A_n$ and $B_n \indica{\tau_1^n > \fg(n)}$.
Let $b_n := \min(1-t_n, \tau_1^n)\le \tau_1$ (note that $b_n \to 1 \wedge \tau_1$ a.s.), choose some $k_0$ such that $t_{k_0}<1/2$, and henceforth assume WLOG that $n\geq k_0$. Then $t_n \leq 1/2$ so that
\begin{align*}
A_n 
\le   \left( { \int_{0}^{b_n} u \, N^{\La,n}(u) \,du  \over \int_0^{b_n} (u+t_n)v(u+t_n) du } + { \int_{b_n}^{\tau_1} uN^{\La,n}(u) du \over \int_{1/2}^1 uv(u)du } \right)
\end{align*}
For $h_1,h_2$ two strictly positive integrable functions over some interval $[a,b]$ we always have that
$
\frac{\int_a^b h_1(u)du}{\int_a^b h_2du} \le \sup_{u \in [a,b]} \frac{h_1(u)}{h_2(u)}.
$
Therefore
\begin{align}
A_n &\le  \sup_{u \in [0, b_n]} \frac{N^{\La,n}(u)}{v(t_n+u)}  + O\left( \int_{b_n}^{\tau_1} u N^{\La,n}(u)du \right)\nonumber \\
&= \sup_{u \in [0, 1]} \frac{N^{\La,n}(u)}{v(t_n+u)}  +  O\left(N^{\La}(1/2) (\tau_1)^2 \right),
\label{EAnbnd}
\end{align}
due to $b_n \le 1$ and $N^{\La,n}(b_n)=\indica{b_n = \tau_1^n<1-t_n} + N^{\La,n}(1-t_n)\indica{b_n = 1- t_n\ge 1/2} \le N^\La(1/2)$, a.s.
Similarly we have
\begin{align*}
 B_n \indica{\tau_1^n > \fg(n)}
 &=   {  \int_{0}^{\tau_1^n} v(u+t_n)du  \over \int_{0}^{\tau_1^n}  N^{\La,n}(u) \,du  } \cdot {  \int_{0}^{1-t_n} v(u+t_n)du  \over  \int_{0}^{\tau_1^n} v(u+t_n)du   } \indica{\tau_1^n > \fg(n)} \\
 &\le \sup_{u \in [0, \tau_1]} \frac{v(t_n+u)}{N^{\La,n}(u)} \cdot {  \int_{t_n}^{1} v(u)du  \over  \int_{t_n}^{t_n+\fg(n)} v(u)du   }.
\end{align*}
Observe that due to $t_n \sim c n^{1-\alpha}$ (claim 1 of Lemma \ref{L:asymptotics}) we have $t_n = o(\fg(n))$ if $\alpha \ge 3/2$, and $t_n/\fg(n) \to c' \in (0,\infty)$ if $\alpha <3/2$. Due to the asymptotic form (\ref{ENasym streg}) for $v$, the sequence
$(\int_{t_n}^{1} v(u)du  /  \int_{t_n}^{t_n+\fg(n)} v(u)du)_{n\geq k_0}$ is uniformly bounded. We conclude that
\begin{align}
\label{EBnbnd}
 B_n \indica{\tau_1^n > \fg(n)} = O\left( \sup_{u \in [0, \tau_1]} \frac{v(t_n+u)}{N^{\La,n}(u)}\right).
 \end{align}
Combining (\ref{Eletit})--(\ref{EBnbnd}), we obtain
\begin{align*}
\frac{Y_n}{\fg(n)}
=O \left[\left(\sup_{u \in [0, 1]} \frac{N^{\La,n}(u)}{v(t_n+u)}  +  N^{\La}(1/2) (\tau_1)^2 \right)  \left( \sup_{s \in [0,\tau_1]}\frac{v(t_n+u)} {N^{\La,n}(u)}  \right)\right] +1, \ \forall n\geq k_0.
\end{align*}
By the Cauchy-Schwarz inequality, it suffices to show that, for all $n$ sufficiently large,  each factor in the brackets above is bounded in $L^4$. This follows immediately from:
\begin{enumerate}
\item[(i)] \cite[Eq. (38)]{bbl1} applied with $s=1$ and $n\geq n_0 \vee k_0$ (see above (33) in \cite{bbl1} for the definition of $r(x;s)$, Lemma 19 in \cite{bbl1} for the choice of $n_0$, and note that this step uses the condition $\La[1-\eta,1]=0$ for some $\eta>0$),
\item[(ii)] $N^\La(1/2) \in L^p$ for all $p\ge 1$ (a consequence of Theorem 2 in \cite{bbl1}).
\item[(iii)] claim 3 of Lemma \ref{L:asymptotics}.
\end{enumerate}
This proves the uniform integrability of $(Y_n/\fg(n))_{n\geq 1}$, and completes the proof of
(\ref{eq:M_n b}).

\subsection{Proof of Theorem \ref{T:A_n}}
\label{S:proof of An}
Recall the construction of Section \ref{S:asym samp},
where the genealogy with mutations is realized for all $n$ simultaneously, with nice monotonicity properties.
In this and the next subsection we will often refer to it under the name the {\em full genealogy (construction or coupling)}. 
Furthermore, we are going to prove theorems \ref{T:A_n} and \ref{T:spectrum} under the additional assumption
$$
\exists \eta>0 \ : \ \La[1-\eta,1]=0,
$$
and we will then explain how this hypothesis can be relaxed in Section \ref{S:relaxing}.

{\em Case $X_n = S_n$.}
For each $s>0$ we have, due to Theorem 5 in \cite{bbl1},
\begin{equation}
\lim_{n\to \infty} \frac{\int_0^s N^{\La,n}(t) \,dt}{\int_0^s v(t_n+t)\, dt} =
1, \mbox{ in probability}.
\label{E:Thm5BBL}
\end{equation}
For the Kingman and the regular variation coalescents
(see Definition \ref{D:regvar}) the above convergence holds almost surely. The total length of
$\bT_n$ is $L_n=\int_0^{\tau_1^n} N^{\La,n}(t) \,dt$.
Observe that $|\int_1^{\tau^n_1} N^{\La,n}(u)du /  \int_0^s v(t_n+t)\, dt|\to 0$ almost surely since in all cases $|\int_1^{\tau^n_1} N^{\La,n}(u)du | \le |\int_1^{\tau_1} N^\La(u) du | <\infty$ a.s., and
$\int_0^s v(t_n+t)\, dt$ diverges in $n$.
Applying \eqref{E:Thm5BBL} with $s=1$, we deduce that
$$
L_n \sim \int_0^1 v(t_n+t)\, dt
$$
in probability (i.e., the ratio of the two sides tends to 1 in probability), and almost surely in the regular variation case.
Using the facts that $v(t_n)=n$ and $v'(q)=-\psi(v(q))$ for all $q>0$,
and applying a change of variables $q=v(t)$,
we obtain
$$
\int_0^1 v(t_n+t)\, dt = 
 \int_{v(1+t_n)}^{n}
 \frac{q}{\psi(q)}\,dq \sim \int_{v(1)}^{n}
 \frac{q}{\psi(q)}\,dq \sim \int_{1}^{n}
 \frac{q}{\psi(q)}\,dq , \mbox{ as }n\to \infty,
$$
since $v(1+t_n) \to v(1)\in(0,\infty)$, and since the integral of $q/\psi(q)$ is finite (resp.~infinite) over $[a,b]$ (resp.~$[a,\infty)$), for all fixed $a,b\in (0,\infty)$.

Recall again the fact that $S_n$ has Poisson ($\theta L_n$) distribution, given $\bT_n$.
Now due to $L_n \to \infty$, almost surely, we obtain
\begin{equation}\label{E:Sn}
 \frac{S_n}
{\displaystyle \int_1^n {q}{\psi(q)^{-1}}\,dq }\longrightarrow \theta,
\end{equation}
in probability, as claimed.
In the regular variation case, this last convergence holds again in the
almost sure sense due the fact that in the full genealogy coupling
$S_n \leq S_{n+1}$, for all $n$, almost surely, and that $\int_1^n {q}{\psi(q)^{-1}}\,dq$ is asymptotic to a multiple of $n^{2-\alpha}$.
To obtain the final claim, we recall (\ref{Epsi strongreg}).
Integrating the RHS and recalling \eqref{E:Sn}, we deduce that $S_n \sim \theta B n^{2-\alpha}$, almost surely, where $B$ is as stated in Theorem \ref{T:A_n}, in consistence with Theorem 1.9 of \cite{bbs1}.
\hfill$\Box$

\medskip
For $X_n=A_n$, our strategy is as follows: we first establish the convergence in probability of $A_n$ in the general case, and then show the almost sure convergence in the strong regular variation case.

\noindent
{\em Case $X_n = A_n$, convergence in probability.}
In the full genealogy construction, we have $A_n \le  S_n+1$ for each $n$, almost surely. Therefore, \eqref{E:Sn} implies that for any $\eps>0$,
\begin{equation}
\label{E A vs L limsup}
\P\left({A_n}\ge (1+\eps) \theta \int_1^n {q}{\psi(q)^{-1}}\,dq\right) \to 0.
\end{equation}
It remains to prove the matching lower-bound.
To do this, for each mutation (or mark) $x$ on $\bT_n$, consider the path $\gamma  = \gamma(x) \subset \bT_n$ defined as follows.
Consider a mutation or mark $x \in \bT_n$ with age $t$. Then $\gamma(x)$ is defined as the path connecting the mark to the leaf carrying the smallest label possible. Since all points of $\gamma$ lie below $x$, the age of any point $y \in \gamma$ is at most $t$.
For example, the $\gamma$ of the mutation encircled in gray on Figure 1 is the path linking it to the leaf labeled by $1$. 
We say that a mark $x$ is \emph{unblocked}
if $\gamma(x)$ carries no other mutation than $x$, and otherwise call it {\em blocked}.
Observe that if $x$ is unblocked then it is guaranteed to contribute one allelic type to $A_n$.
Intuitively, it is rather likely that $\gamma(x)$ is unblocked.
Indeed, since the age $M_n$ of a randomly chosen point on $\bT_n$ is typically small, then $e^{-\theta M_n} \approx 1-\theta M_n$, so the probability that a typical mutation is blocked is of order
$\theta \E (M_n) \to 0$. This suggests that the proportion of blocked mutations is negligible, which is sufficient to yield the desired result.

More rigorously, given  $ \bT_n$ and $S_n$, the mutations fall on $\bT_n$ as $S_n$ i.i.d.~uniformly chosen random points.
For $1\le i \le S_n$, let $K_{i,n}$ be the ``good'' event that the $i$th mutation
is unblocked, and define
\[
Y_n:=\sum_{i=1}^{S_n} \indic{K_{i,n}},
\]
the total number of unblocked mutations. As already argued, we have $Y_n \leq A_n \leq S_n+1$ almost surely, so in view of \eqref{E:Sn} it suffices to prove
\begin{equation}
\label{E S vs Y}
\lim_{n\to \infty}\frac{Y_n}{S_n} = 1,\mbox{ in probability}.
\end{equation}
Note that, given $\bT_n$ and $S_n$, the events $(K_{i,n})_{i=1,\ldots, S_n}$ are exchangeable.
In particular, almost surely,
\[
\P(K_{1,n}| \bT_n, S_n)= \P(K_{i,n}| \bT_n, S_n), \ i=1,\ldots, S_n.
\]
Note in addition that the age of the mutation corresponding to $K_{1,n}$ is equal in distribution to $M_n$
from Proposition \ref{P:M_n}.
Due to the above discussion, we have
\begin{equation}
\label{E: K1N}
\P (K_{1,n} \vert \bT_n, S_n, M_n) = \left( 1-\frac{M_n}{L_n}\right)^{(S_n-1)_+}
\end{equation}
almost surely on the event $\{S_n>0 \}$. We extend the definition of $K_{1,n}$ using the above on the complement $\{S_n= 0\}$,
making $K_{1,n}$ certain in this case.
Fix $\eps>0$ and note that by Markov's inequality, we have
\[
\P( Y_n \le (1-\eps) S_n| \bT_n, S_n) = \P( S_n - Y_n \ge \eps S_n|\bT_n, S_n) \leq \frac{\E(S_n - Y_n|\bT_n, S_n)}{\eps S_n},
\]
with the convention $0/0=1$. Therefore, due to the above discussion and (\ref{E: K1N}), we obtain
\begin{align}
\P(Y_n \le (1-\eps) S_n)  &\le \frac1\eps \E\left[ \frac{\E(S_n- Y_n \vert  \bT_n, S_n)}{S_n} \right] =\frac1\eps\E\left[    \P(K_{1,n}^c  \vert \bT_n, S_n,M_n) \right]\nonumber\\
&\le \frac1\eps\E\left[1- \left( 1-\frac{M_n}{L_n}\right)^{(S_n-1)_+}  \right].\label{bdprob}
\end{align}
The random variable (conditional probability) in this last expectation is bounded by $1$, almost surely.
Therefore, in order to show that it converges to $0$ in the mean (in $L^1$), it suffices to show that it converges to $0$ in probability.
Now note that, since $1-(1-x)^n \leq nx$ for $n\in \N$ and $x\geq 0$,
\begin{equation}
\label{Eupperbd easy}
1-\left( 1-\frac{M_n}{L_n}\right)^{(S_n-1)_+} \leq 2\theta M_n \indica{S_n/L_n \leq 2\theta} + \indica{S_n/L_n > 2\theta}
\end{equation}
So, for a fixed small $\delta >0$, we have
\begin{equation}
\label{E nicebound}
\P\left(1-\left( 1-\frac{M_n}{L_n}\right)^{(S_n-1)_+} >\delta\right) \leq \P( 2\theta M_n > \delta) + \P(S_n > 2 \theta L_n).
\end{equation}
Due to Proposition \ref{P:M_n} (a), the first term on the RHS in (\ref{E nicebound}) vanishes as $n\to \infty$, and since
$S_n$ has Poisson (rate $\theta L_n$) distribution, given $L_n$, the second term also vanishes.
Therefore (\ref{bdprob}) converges to $0$ as $n\to \infty$, implying (\ref{E S vs Y}).
\hfill$\Box$

\noindent
{\em Case $X_n=A_n$ with strong $\alpha$-regular variation.}
Here we use a variation of (\ref{Eupperbd easy}):
\begin{equation}
\label{E upperbd intric}
1-\left( 1-\frac{M_n}{L_n}\right)^{(S_n-1)_+} \leq \left( M_n \cdot \frac{S_n}{L_n} \right) \wedge 1 \leq
2\theta M_n \indica{S_n/L_n \leq 2\theta} + \left(M_n \wedge \frac{1}{2\theta}\right) \cdot \frac{S_n}{L_n} \indica{S_n/L_n > 2\theta}.
\end{equation}
Therefore, applying the Cauchy-Schwarz inequality,
\[
\E\left[1-\left( 1-\frac{M_n}{L_n}\right)^{(S_n-1)_+}\right] \leq 2 \theta \E(M_n) + \sqrt{\E[ (M_n \wedge 1/(2\theta))^2] \cdot \E[(S_n/L_n)^2]}.
\]
Due to Proposition \ref{P:M_n} (b), the first term above is $O(n^{-2\delta})$ for some $\delta >0$.
For the second one, note that $\E[ (M_n \wedge 1/(2\theta))^2]\leq \E(M_n)/(2\theta) = O(n^{-2\delta}/\theta)$ and that $\E[(S_n/L_n)^2] = O(\theta^2\vee \theta)$.
Indeed, since $S_n$ is Poisson (rate $\theta  L_n$), given $L_n$, we have (assuming $n\geq 3$)
$$
\E[(S_n/L_n)^2] = \E[\E(S_n^2\vert L_n) /L_n^2] = \E[\theta^2 +\theta/L_n] \leq \theta^2+\theta\E(1/L_3),
$$
where it is simple to verify that $\E(1/L_3) <\infty$.
As a consequence, $\E(1-\left( 1-M_n/{L_n}\right)^{(S_n-1)_+})) = O(n^{-\delta})$.
Due to (\ref{bdprob}), we deduce
\begin{equation}
\label{bdpoly}
\P(Y_n \le (1-\eps) S_n) \le \frac{c n^{-\delta}}{\eps}
\end{equation}
for some $\delta>0$, and $c<\infty$ which depends only on $\theta$. Consider the subsequence $n_k = \lfloor k^{2/\delta}\rfloor$, $k\geq 1$.
By the Borel-Cantelli lemma and \eqref{bdpoly}, $Y_n/S_n$ tends to 1 along the subsequence $(n_k)$.
Moreover, since both $A_n \leq A_{n+1}$ and $S_n \leq S_{n+1}$, for all $n$, almost surely,
we have
\begin{equation}\label{E:convas}
\frac{S_{n_{k}}}{S_{n_{k+1}}}\frac{A_{n_k}}{S_{n_k}} \leq \frac{A_n}{S_n} \leq \frac{A_{n_{k+1}}}{S_{n_{k+1}}}\frac{S_{n_{k+1}}}{S_{n_{k}}} \mbox{ whenever } n\in [n_k, n_{k+1}] \mbox{ for some $k\geq 1$. }
\end{equation}
Since we already verified at the beginning of the argument that $S_{n_k}/S_{n_{k+1}} \sim (n_k/n_{k+1})^{2-\alpha}$, almost surely, and since $(n_k/n_{k+1})^{2-\alpha} \to 1$, as $k\to \infty$,
the almost sure convergence along the subsequence $(n_k)_{k\geq 1}$ and (\ref{E:convas}) imply that $A_n/S_n \to 1$ almost surely. This finishes the proof of Theorem \ref{T:A_n}.
\hfill$\Box$
\begin{remark}
\label{R:xi case}
As already mentioned, the above convergence in probability for $X_n=A_n$ is
proved in~\cite{sampl_xi} for a more general class of regular $\Xi$-coalescents, using
a compact martingale argument that accounts for all the mutations in a dynamic
way (from the point of view of coalescent evolution), which can be easily extended to handle randomly (and nicely) varying mutation rates. However, that approach is not well-suited for obtaining qualitative or quantitative information about a random (typical) mutation.
The present approach could be used even in the setting without martingale structure, once
given the estimates in the form of Proposition \ref{P:M_n}.
Furthermore, the random mutation analysis enables us to easily identify
the asymptotic behavior of $M_{k,n}$ with that of $F_{k,n}$ (see the end of the proof of Theorem \ref{T:spectrum}).
\end{remark}

\subsection{Proof of Theorem \ref{T:spectrum}}
\label{S:proof of T:spectrum}
Recall the setting of Theorem \ref{T:spectrum}.
We first concentrate on the result (\ref{E:allelic spectrum}) in the case of the allelic partition, which we restate here for
convenience: if
$F_{k,n}$ denotes the number of allelic types in the allelic partition carried by exactly $k$ individuals, then for any fixed $
k\ge 1$,
\begin{equation}
\label{E:restatF}
\frac{F_{k,n}}{n^{2-\alpha}} \to \theta B (2-\alpha) \frac{(\alpha -1) \ldots (\alpha +k-3)}{k!}, a.s.
\end{equation}
as $n \to \infty$.
The key to proving \eqref{E:restatF} is to apply Corollary 21 in \cite{ghp},
which could be thought of as a Tauberian theorem for random exchangeable partitions, that establishes the mutual equivalence between the strong almost sure asymptotics
\eqref{E:TAn-strong}, \eqref{E:allelic spectrum}, and \eqref{E:Tfreq}.

We now recall the setting in~\cite{ghp}. Let $\vec{p}=(p_1,p_2,\ldots)$ be a deterministic sequence
such that
$p_1\geq p_2\geq \ldots \geq 0$ and $\sum_i p_1=1$.
Suppose that $\Theta =\Theta_{\vec{p}}$ is an exchangeable random partition on $\N$, obtained by performing
the paintbox construction generated by $\vec{p}$ (see, e.g.,~\cite{aldous_stflour} or Definition 1.2 in~\cite{ensaios}). Let $\Theta^{n}$
denote the restriction of $\Theta$ onto $[n]=\{1, \ldots, n\}$.
Let $K_n$ be the number of blocks in $\Theta^{n}$, and for each $r=1,\ldots,n$, let
$K_{n,r}$ be the number of blocks in $\Theta^{n}$ containing exactly $r$ elements.
The frequency vector $\vec{p}$ is said to be {\em regularly varying with index $\gamma$} if
$$
\sum_i \indica{p_i \ge x} \sim \ell(1/x) x^{-\gamma}
$$
as $x \to 0$, where $\ell$ is a slowly varying function.

\begin{lemma} \emph{(Corollary 21 in~\cite{ghp})} There is equivalence between the following statements.

\begin{enumerate}
\item[(a)] $\vec{p}$ is regularly varying with index
$\gamma$.

\item[(b)] $K_n \sim \Gamma(1-\gamma)n^{\gamma} \ell(n)$, almost surely as $n \to \infty$.

\end{enumerate}
If either (a) or (b) holds, then for each fixed $r\ge 1$,
$$
K_{n,r} \sim \frac{\gamma\Gamma(r-\gamma)}{r!} n^\gamma \ell(n), \mbox{ almost surely.}
$$
\end{lemma}

We refer the reader to Theorem 1.11 in~\cite{ensaios} for an overview and a sketch of proof, and to Schweinsberg~\cite{Schweinsberg2009} for a version of this result where the assumptions and conclusions are convergence in probability, rather than almost surely.

\begin{proof}[Proof of Theorem \ref{T:spectrum}] We apply the above lemma to the
allelic partition $\Theta$, which is an exchangeable random
partition. As the reader is about to see, for this particular application  the almost sure convergence
in \eqref{E:TAn-strong} of Theorem \ref{T:A_n} is crucial.
Since $\Theta$ is random exchangeable, by Kingman's representation theorem (Theorem 1.1 in~\cite{ensaios}),
all the blocks of $\Theta$ have a well-defined asymptotic frequency. We let $\vec{P}$ be the sequence of block
frequencies
in decreasing order. Thus $\vec{P} \in \nabla_{\leq 1}=\{(p_1,p_2\ldots,):p_1\geq p_2\geq \ldots \geq 0,\sum_{i=1}^\infty p_i\leq 1\}$.
Moreover, given $\vec{P}$, $\Theta$
has the law of a paintbox partition derived from $\vec{P}$.
Note that $A_n$ then corresponds to the total number of blocks of $\Theta^n$, while $F_{k,n}$ is the number of blocks of size exactly $k$.
 In particular, since $A_n =o(n)$ almost surely, it
must be that $\P(\vec{P} \in \nabla_1)=1$, where
$\nabla_1=\{(p_1,p_2\ldots,):p_1\geq p_2\geq \ldots \geq 0,\sum_{i=1}^\infty p_i=1\}$.
That is, $\Theta$ has no singletons (or no ``dust") almost surely. Moreover, since
$\P(A_n \sim \theta Bn^{2-\alpha})=1$
by \eqref{E:TAn-strong},
then also
$$\P(A_n \sim \theta Bn^{2-\alpha}|\vec{P})=1, a.s.$$
Therefore, Corollary 21 in~\cite{ghp} implies
that
$$
\P(\vec{P}\mbox{ is regularly varying with index }2-\alpha)=1,
$$
and, moreover that, if $N(x) = \sum_{i\ge 1} \indic{\{P_i \ge x\}}$, then
$$
N(x) \sim \frac{\theta B}{\Gamma(\alpha -1)} x^{\alpha -2}, \mbox{ almost surely}.
$$
Furthermore,
$$
\P\left(\left.F_{k,n} \sim \frac{\theta B\cdot (2-\alpha)\Gamma(k-2+\alpha)}{r! \Gamma(\alpha-1)} n^{2-\alpha}\right|\vec{P}\right)=1, a.s.
$$
Taking expectations in the last identity yields (\ref{E:restatF}).

It remains us to prove (\ref{E:allelic spectrum}) in the case
where $X_{k,n} = M_{k,n}$, the number of genetic types under the infinite sites model.
Observe another
important property of our full genealogy coupling
(cf.~Figure 1): a family of size $k$ in the infinite allele model necessarily descends from the same mutation, and therefore this mutation affects at least $k$ leaves.
Thus, for all $n\ge 1$, and for all fixed $k\in\{1,\ldots,n\}$,
we have that
\begin{equation}\label{E:F<M}
\bar F_{k,n} \le \bar M_{k,n},
\end{equation}
where $\bar F_{k,n} = \sum_{ j = k }^n F_{j,n}$ and $\bar M_{k,n} = \sum_{j=k}^n M_{j,n}$ are the cumulative number of
families of size larger or equal to $k$. Let
$$
c_k= (2-\alpha) \frac{(\alpha -1) \ldots (\alpha +k-3)}{k!}
$$
so that $F_{k,n} \sim n^{2-\alpha}\theta B c_k.$ Observe that $\bar F_{1,n} = A_n$ and thus  (since $\bar F_{k+1,n} = \bar F_{k,n} -  F_{k,n}$) we deduce by induction on $k\ge 1$ that   $\bar F_{k,n}\sim n^{2-\alpha} \theta B \bar c_k$,
where $\bar c_k = \sum_{j=k}^\infty c_j$.
Here we use the fact $\bar c_1= 1$ (see~\cite{bbs2}, Lemma 30 or~\cite{saintflour}, display (3.38)).

Therefore, Theorem \ref{T:spectrum} will be proved, provided we show that, for each fixed $k\ge 1$,
\begin{equation}\label{E:goalfreq}
\bar M_{k,n} \le \bar F_{k,n} + o(n^{2-\alpha}),
\end{equation}
almost surely as $n\to \infty$.
This can be done by the following adaptation of the argument for Theorem \ref{T:A_n}.

Fix $k\ge 1$.
We extend the definition of an unblocked mutation as follows.

Recall that any point $x\in \bT_n$ corresponds uniquely to a block $B$ of the coalescing partition $\Pi_t$, where $t$ is the age of $x$. Suppose that $B=\{i_1 < \ldots < i_m\}$, for some $m\in \N$, and define
$T(x) \subset \bT_n$ to be the restriction of the coalescence subtree generated by the paths that lead from $x$ to the leaves labeled by $i_1, \ldots, i_{m\wedge k}$.
For example, for the mutation encircled in black on Figure 1, this subtree has four leaves labeled by $\{2,3,5,6\}$.
Note furthermore that, for $k=1$, $T(x)$ coincides with the path $\gamma(x)$ defined in the proof of Theorem \ref{T:A_n}, and that the total length of $T(x)$ cannot exceed $k t$.

Let us say that $x \in \bT_n$ is \emph{$k$-unblocked} if $T(x)$ carries no other mark than $x$, and otherwise call it {\em $k$-blocked}.
Similarly to the proof of Theorem \ref{T:A_n}, define $\bar K_{i,n}$ as the event that the $i$th mutation
(picked at random without replacement) is $k$-unblocked, and let $\bar Y_{k,n}:=\sum_{i=1}^{S_n} \indic{\bar K_{i,n}}$.
Reasoning as for \eqref{bdprob}, and using the fact that the length of $T$ which corresponds to the randomly
picked mutation is at most $kM_n$, we obtain, for all fixed $\eps>0$,
\begin{align*}
\P(\bar Y_{k,n}\le (1-\eps) S_n) & \le \frac1\eps \E\left[ 1- \left( 1-\frac{kM_n}{L_n}\right)^{(S_n-1)_+}\right].
\end{align*}
Since $k$ is fixed, the bound (\ref{E upperbd intric}) with $M_n$ replaced by $k M_n$ will lead to
\begin{equation}\label{bdpoly3}
\P(\bar Y_{k,n}\le (1-\eps) S_n) \le \frac{c k n^{-\delta}}{\eps},
\end{equation}
where $\delta$ is as in \eqref{bdpoly}, and $c$ depends only on $\theta$.
Therefore,
we have as before $\bar Y_{k,n}/S_n\to 1$ almost surely along the subsequence $(n_j)$, where
 $n_j = \lfloor j^{2/\delta}\rfloor$, $j\geq 1$.
In particular, $S_{n_j} - \bar Y_{k,n_j} = o(n_j^{2-\alpha})$, $j\geq 1$.

Denote by $\bar M_{k,n}'$ the number of $k$-unblocked mutations that span {\em at least} $k$ leaves. For each such mutation, the corresponding family in the allelic partition is of size at least $k$, so $\bar M_{k,n}' \le \bar F_{k,n}$. Moreover,
$$
0\le \bar M_{k,n} - \bar M_{k,n}' \le S_n - \bar Y_{k,n}
$$
since $S_n - \bar Y_{k,n}$ accounts for all the $k$-blocked mutations, even if they span fewer than $k$ leaves. Thus, due to the previous observations,
\begin{equation}\label{E:MM'}
\bar M_{k,n_j}' = \bar M_{k,n_j} -  o(n_j^{2-\alpha}) \leq \bar F_{k,n_j} \leq \bar M_{k,n_j}, \ j \geq 1,
\end{equation}
implying (\ref{E:goalfreq}) with $n_j$ in place of $n$, and in particular $\bar M_{k,n} \sim \bar F_{k,n}$,  along the subsequence $(n_j)_{j\geq 1}$.
Since we already know that $\bar F_{n,k} \sim \theta B \bar c_k n^{2-\alpha}$ almost surely, and since
both $(\bar F_{n,k})_{n\geq 1}$ and $(\bar M_{n,k})_{n\geq 1}$ are non-decreasing, almost surely,
reasoning as in \eqref{E:convas} gives
$$
\bar M_{k,n} \sim \theta B \bar c_k n^{2-\alpha},
$$
as $n\to \infty$, which finishes the proof of Theorem \ref{T:spectrum}.
\end{proof}

Here is an interesting consequence about the structure of $\bT_n\cap\cP$,
Let $\#_\ell T(x)$ be the number of leaves of $\bT_n$ contained in $T(x)$.
If $x$ is a randomly chosen mutation, then denote its $\#_\ell T(x)$ simply by $\#_\ell T$.
\begin{coro} We have
$
\lim_{k\to \infty} \lim_{n\to \infty} \P(\#_\ell T \geq k)=0.
$
\end{coro}
This claim is weaker than the statement that $\#_\ell T$ is stochastically bounded.
\begin{proof}
Given $\bT_n,\cP$,
the probability that $\#_\ell T \geq k$ equals precisely $\bar M_{k,n}/S_n$.
Due to Theorem \ref{T:spectrum}, the almost sure limit of this is $\bar c_k$, defined in the proof above.
Since $\bar M_{k,n}/S_n\in [0,1]$, this convergence is also in $L^1$, and the corollary follows due to
$\lim_k\bar c_k=0$.
\end{proof}

\subsection{Relaxing the condition $\La[1-\eta,1]=0$}
\label{S:relaxing}
Define $\La_\eta(dx):= \La(dx) \indic{[0,1-\eta]}(x)$.
Then it is easy to see (or consult, e.g.,~\cite{bbl1}) that
there exists a path-wise full genealogy coupling of the corresponding $\La$-coalescent and $\La_\eta$-coalescent,
so that, almost surely,
for each $n\geq 1$,
$N^{\La,n}(t) = N^{\La_\eta,n}(t)$, for all $t\in [0, T_\eta]$, and
$N^{\La,n}(t) \leq N^{\La_\eta,n}(t)$ for all $t> T_\eta$, where $T_\eta$ is an exponential (rate $\int_{(1-\eta,1]} 1/x^2\,\La(dx)<\infty$) random variable.
By the ``full genealogy'' coupling, we also mean that the same realization of the mutation process $\cP$ is used for both the
restriction of $\bT_n^\La$ onto $[0,T_\eta]$, and the restriction of $\bT_n^{\La_\eta}$ onto $[0,T_\eta]$, simultaneously for all $n$.

The crucial fact is that then the family of non-negative random variables
\[
\max\left\{\int_{T_\eta}^{1\vee \tau_1^{n;\La}} N^{\La,n}(t)\,dt,\,
\int_{T_\eta}^{1\vee \tau_1^{n;\La_\eta}} N^{\La_\eta,n}(t)\,dt\right\}, \ n \geq 1,
\]
is bounded from above by a finite random variable, almost surely.
Therefore, $L_n^\La \sim L_n^{\La_\eta}$ almost surely, as well as, $S_n^\La \sim S_n^{\La_\eta}$
and $A_n^\La \sim A_n^{\La_\eta}$, again almost surely, as $n\to \infty$.
This implies Theorem \ref{T:A_n} in the general case.
And similarly, in the above coupling, we have
 $F_{n,k}^\La \sim F_{n,k}^{\La_\eta}$ and $M_{n,k}^\La \sim M_{n,k}^{\La_\eta}$, for each fixed $k$, almost surely as $n\to \infty$, yielding Theorem \ref{T:spectrum}.

\end{document}